\documentclass[reqno,a4paper,11pt]{amsart}
\usepackage{color}
\usepackage{amsmath,amssymb,amsfonts,amsthm,enumerate}
\usepackage{mathrsfs}
\usepackage[T1]{fontenc}
\usepackage[active]{srcltx}
\usepackage{epsfig}
\usepackage{graphicx}

\numberwithin{equation}{section} 
\newtheorem{theorem}{Theorem}[section]
\newtheorem{lemma}[theorem]{Lemma}

\newtheorem{remark}[theorem]{Remark}
\newtheorem{proposition}[theorem]{Proposition}

\newcommand{\N}{\mathbb{N}}
\newcommand{\R}{\mathbb{R}}

\newcommand{\Lin}{\mathscr{L}}

\renewcommand{\H}{\mathcal{H}}

\newcommand{\PP}{\mathscr{P}}

\newcommand{\e}{\varepsilon}
\newcommand{\g}{\gamma}

\renewcommand{\l}{\lambda}

\newcommand{\p}{\partial}

\newcommand{\Id}{\operatorname{Id}}

\newcommand{\tD}{{\tilde D}}

 \newcommand{\tauV}{{\kern-3pt\tau}}

 \newcommand{\oVVVk}{\overline{\mbox{\boldmath$V$}}\kern-3pt}
 \newcommand{\tVVVk}{\tilde{\mbox{\boldmath$V$}}\kern-3pt}

\begin{document}

\title[Free boundary problems for fully nonlinear parabolic equations]{A general class of free boundary problems\\ for fully nonlinear parabolic equations}

\author[Alessio Figalli]{Alessio Figalli}
\address{Mathematics Department, The University of Texas at Austin,  Austin, Texas, 78712-1202, USA}
\email{figalli@math.utexas.edu}

\author[Henrik Shahgholian ]{Henrik Shahgholian}
\address{Department of Mathematics, KTH Royal Institute of Technology, 100~44  Stockholm, Sweden}
\email{henriksh@kth.se}

\thanks{A. Figalli was partially supported by NSF grant DMS-1262411.
}

\begin{abstract}    
In this paper we consider the fully nonlinear parabolic free boundary problem 
$$
\left\{
\begin{array}{ll}
F(D^2u) -\partial_t u=1 & \text{a.e. in }Q_1 \cap \Omega\\
|D^2 u| + |\partial_t u| \leq K & \text{a.e. in }Q_1\setminus\Omega ,
\end{array}
\right.
$$
where $K>0$ is a positive constant, and $\Omega$ is an (unknown) open set.

Our main result is the optimal regularity for solutions to this problem:
namely, we prove that $W_x^{2,n} \cap W_t^{1,n} $ solutions are locally $C_x^{1,1}\cap C_t^{0,1}  $ inside $Q_1$.
A key starting point for this result is a new BMO-type estimate which extends to the parabolic setting
the main result in \cite{CH}.

Once optimal regularity for $u$ is obtained, we also show regularity for the free boundary $\partial\Omega\cap Q_1$
under the extra condition that $\Omega \supset \{ u \neq 0 \}$, and  a uniform thickness assumption on the coincidence set $\{ u =  0 \}$,
\end{abstract}

\maketitle


\section{Introduction and main result}

\subsection{Setting of the problem}
In this paper we will use $Q_r(X):=B_r(x)\times (t-r,t)\subset \R^n\times \R$ to denote the parabolic ball of radius $r$ centered at a point $X=(x,t) \in \R^{n+1}$,
and we will use the notation $Q_r=Q_r(0)$.

Our starting point will be a $W_x^{2,n}(Q_1)\cap W_t^{1,n}(Q_1)$ function $u:Q_1 \to \R$
 satisfying
\begin{equation}\label{eq:obstacle}
\left\{
\begin{array}{ll}
\H (u)=1 & \text{a.e. in }Q_1 \cap \Omega ,\\
|\tD^2 u|  \leq K & \text{a.e. in }Q_1\setminus\Omega ,
\end{array}
\right.
\end{equation}
where $\tD^2 u = (D_x^2 u, D_t u) \in \R^{n^2 +1}$, $\H ( u):= F(D^2u)- \partial_t u$, 
$K>0$, and $\Omega \subset \R^{n+1}$ is some unknown open set.
Since, by assumption, $\tD^2u$ is bounded in the complement of $\Omega$, we see that $\H(u)$
is bounded inside the whole $Q_1$ and $u$ is a so-called ``strong $L^n$ solution''
to a fully nonlinear parabolic equation with bounded right hand side \cite{CrKS}.
We refer to \cite{Wang1,CrKS} as basic references to parabolic fully nonlinear equations and viscosity methods.

The above free boundary problem has a very general form and encompasses several 
other free boundaries of obstacle type. In the elliptic case, it has been recently studied by the authors in
 \cite{FS}. We also refer to several other articles concerning similar type of problems: For elliptic case see \cite{CKS}, \cite{ALS1}, and for parabolic case see
 \cite{CPS}, \cite{ALS2}. One may find applications and relevant discussions about these kinds of problems in these articles.
 
Since most of the result follow the same line of arguments (sometimes with obvious modifications) as that of its elliptic counterpart done in \cite{FS},
here we have chosen not to enter into the details of the proof as they can be worked out in a similar way as in the elliptic case.
Instead, we shall give the outline of the proofs and point out all the necessary changes. 
 For the reader unfamiliar with these techniques, we suggest first to read \cite{FS}.\\

Going back to our problem, we observe that, if $u \in W_x^{2,n}\cap W_t^{1,n}$, then $\tD^2u=0$ a.e. inside  $\{u=0\}$, and $D^2u=0$ a.e. inside  $\{\nabla u=0\}$. In particular
we easily deduce that \eqref{eq:obstacle} includes, as special cases, both $\H(u)=\chi_{\{u \neq 0\}}$
and $\H(u)=\chi_{\{\nabla u \neq 0\}}$.

We assume that:
\begin{enumerate}
\item[(H0)] $F(0)=0$.
\item[(H1)] $F$ is uniformly elliptic with ellipticity constants $0<\l_0\leq \lambda_1<\infty$, that is,
$$
\PP^-(P_1-P_2)
 \leq F(P_1) - F(P_2) \leq \PP^+(P_1-P_2)
$$
for any $P_1,P_2$ symmetric, where $\PP^-$ and $\PP^+$ are the extremal Pucci operators:
$$
\PP^-(M):=\inf_{\lambda_0 \Id \leq N \leq \lambda_1 \Id} {\rm trace}(NM),\qquad \PP^+(M):=\sup_{\l_0 \Id \leq N \leq \l_1 \Id} {\rm trace}(NM).
$$
\item[(H2)] $F$ is either convex or concave. 
\end{enumerate}

Under assumptions (H0)-(H2) above, strong $L^n$ solutions are also viscosity solutions \cite{CKS},
and hence regularity results for parabolic fully nonlinear equations \cite{Wang1,Wang2} show that $u \in W_x^{2,p}(Q_\rho)\cap W_t^{1,p}(Q_\rho)$ for all $\rho\in (0,1)$ and $p <\infty$.

Our first result concern the optimal $C_x^{1,1}\cap C_t^{0,1}$-regularity for $u$.
Once this will be done, we will be able to study the regularity of the free boundary.

\subsection{Main results}
Our first result concerns the optimal regularity of solutions to \eqref{eq:obstacle}:

\begin{theorem} (Interior $C_x^{1,1}\cap C_t^{0,1}$ regularity)
\label{thm:C11}
Let $u:Q_1 \to \R$ be a $W_x^{2,n} \cap W_t^{1,n}$ solution of \eqref{eq:obstacle}, and
assume that $F$ satisfies (H0)-(H2).
Then there exists a  constant $\bar C =\bar C(n,\lambda_0,\lambda_1,\|u\|_\infty)>0$ such that
$$
|\tD^2u| \leq \bar C, \qquad \hbox{in } Q_{1/2}.
$$
\end{theorem}

To state our result on the regularity of the free boundary, we need to introduce the concept of minimal diameter:
 Set $\Lambda:=Q_1\setminus \Omega$, and for any set $E \in \R^n$ let
$\operatorname{MD}(E)$ denote  the smallest possible distance between two parallel hyperplanes containing $E$. Then, given
a point $X^0=(x^0,t^0) \in \R^{n+1}$, we define
$$
\delta_r(u,X^0):=\inf_{t \in [t_0-r^2,t_0+r^2]}\frac{\operatorname{MD}\bigl(\Lambda\cap \bigl(B_r(x^0) \times \{t\}\bigr)\bigr)}{r}.
$$
In other words, $\delta_r(u,X^0)$ measures the thickness of the complement of $\Omega$ at all  time levels  $t\in (t^0-r^2,t^0+r^2)$, around the point $x^0$.
Notice that $\delta_r$ depends on $u$ since $\Omega$ does.
In particular, we observe that if $u$ solves \eqref{eq:obstacle} for some set $\Omega$,
then  $u_r(y,\tau ):=u(x+ry, t + r^2\tau)/r^2$ solves \eqref{eq:obstacle} with
$$
\Omega_r:=\{(y,\tau)\,:\,(x+ry, t + r^2\tau) \in \Omega\}
$$
in place of $\Omega$,
and $\delta_r$ enjoys the scaling property $\delta_1(u_r,0)= \delta_r(u,X)$, $X=(x,t)$.

Our result provides regularity for the free boundary under a uniform thickness condition.
As a corollary of our result, we deduce that Lipschitz free boundaries
are $C^1$, and hence smooth \cite{F}.

\begin{theorem} (Free boundary regularity)\label{thm:C12}
Let $u:Q_1 \to \R$ be a $W_x^{2,n} \cap W_t^{1,n}$ solution of \eqref{eq:obstacle}. 
Assume that $F$ is convex and satisfies (H0)-(H1), and that 
$\Omega\supset\{u \neq 0\}$.
Suppose further that there exists $\e>0$ such that
$$
\delta_r(u,z) > \e
\qquad \forall \,r<1/4,\,z\in \partial \Omega \cap Q_{r}(0).
$$ 
Then 
$\partial \Omega \cap Q_{r_0}(0)$  is  a  $C^1$-graph in space-time,
 where $r_0$ depends only on $\e$ and the data.
\end{theorem}

The paper is organized as follows:

In Section~\ref{regularity} we prove Theorem~\ref{thm:C11}.
Then in Section~\ref{sect:non deg} we investigate the non-degeneracy of solutions, and classify global solutions under a suitable thickness assumption.
In  Section~\ref{dir-mon} we show directional monotonicity for local solutions
which gives Lipschitz (and then $C^1$) regularity for the free boundary, as shown in Section \ref{Local-regularity}.


\section{Proof of Theorem~\ref{thm:C11}}\label{regularity}

The proof of this theorem follows the same line of ideas as its elliptic counterpart
\cite{FS}. First one starts from a BMO-type estimate on $D^2u$,
and then one shows a dichotomy that either $ u$ has quadratic growth 
away from a  free boundary point $X^0$, or the density of the set $\Lambda $ at $X^0$ vanishes fast enough to assure the quadratic bound.

In \cite{FS} the following result was a consequence of the BMO-type estimate proved in \cite{CH}.
Since we could not find a reference for this estimate in the parabolic case, we prove this result in the appendix.
We notice that our proof is much simpler than the one in \cite{CH} and actually gives a new proof of the results there (see Remark \ref{cor:bmo}).

With no loss of generality, we will carry out the  proof   at the origin, by letting $X^0=(0,0)$. 

\begin{lemma}
\label{lem:bounded}
There exist a  constant $C=C(n,\lambda_0,\lambda_1,\|u\|_\infty)$, and a family of second order  parabolic polynomial $\{P_r\}_{r \in (0,1)}$ solving $\H(P_r)=0$,
 such that
 \begin{equation} \label{eq:uPr}
 \sup_{Q_r(0)}|u -P_r| \leq C r^2, \qquad \forall \ r \in (0,1).
 \end{equation}
Consequently
\begin{equation}
\label{eq:ur2}
\sup_{Q_r(0)}|u| \leq (C r^2+| P_r|) , \qquad \forall \ r \in (0,1).
\end{equation}
\end{lemma}

 The first statement in the Lemma is proven in Appendix (see \eqref{pointwise-bmo} and Lemma \ref{P_k-close} there),
 while the second estimate is a straightforward consequence of the first one. It should be remarked that these polynomials $P_r$ need not to be unique.

 Define
 \begin{equation}
  \label{eq:Ar}
  A_r:=\{(x,t): \ (rx,r^2t) \in Q_r\setminus \Omega\} \subset Q_1 \qquad  \forall \  r<1/4 .
  \end{equation}
We shall  prove that if $|P_r|$ is sufficiently large then the measure of $A_r$ 
has to decay  geometrically.

\begin{proposition}
\label{prop:M}
Set $\tilde P_r:= \tilde D^2 P_r$.
There exists $M>0$ universal such that, for any $r \in (0,1/8)$, if $|\tilde P_r| \geq M$ then
$$
|A_{r/2}| \leq \frac{|A_r|}{2^{n+1}}.
$$
\end{proposition}

The proof of the proposition follows the same lines of ideas as that of \cite[Proposition 2.4]{FS}.
However, since the changes are not completely straightforward, for the reader's convenience we present the proof here.

\begin{proof}
Set $u_r(y,t):=u(ry,r^2)/r^2$,  $\H_r(v):=F(D^2 P_r+ D^2v) -\partial_t P_r -\partial_t v$, and let
\begin{equation}\label{rewrite}
u_r(y,t)=P_r (y,t) + v_r(y,t)+w_r(y,t),
\end{equation}
where $v_r$ is defined as the solution of
\begin{equation}
\label{eq:vr}
\left\{
\begin{array}{ll}
\H_r(P_r+ v_r)-1=0 & \text{in } Q_1,\\
v_r(y,t)=u_r(y,t)-P_r ( y,t) & \text{on }\p_p Q_1,
\end{array}
\right.
\end{equation}
where $\p_pQ_1$ denotes the parabolic boundary of $Q_1$, 
and by definition $w_r:=u_r-P_r  -v_r$.

Set $f_r:=\H(D^2u_r) \in L^\infty(B_1)$ (recall that $|\tD^2 u_r|\leq K$ a.e. inside $A_r$).
Notice that, since $f_r=1$ outside $A_r$,
$$
\H(u_r)-\H(P_r+v_r) =(f_r-1)\chi_{A_r},
$$
so it follows by (H1) that $w_r$ solves
\begin{equation}
\label{eq:wr}
\left\{
\begin{array}{ll}
\PP^-(D^2w_r) - \partial_t w_r \leq (f_r-1)\chi_{A_r} \leq \PP^+(D^2w_r)  -\partial w_r & \text{in }Q_1,\\
w_r=0& \text{on }\p_p Q_1.
\end{array}
\right.
\end{equation}
Hence, we can apply the ABP estimate \cite[Theorem 3.14]{Wang1} to deduce that
\begin{equation}
\label{eq:ABPwr}
\sup_{Q_1}|w_r| \leq C \| \chi_{A_r}\|_{L^{n+1}(Q_1)} = C|A_r|^{1/(n+1)}.
\end{equation}
Also, since $\H(P_r)=0$ and $v_r$ is universally bounded on $\partial_p Q_1$ (see \eqref{eq:uPr} and \eqref{eq:vr}), by the parabolic Evans-Krylov's theorem \cite{Kr}
applied to \eqref{eq:vr} we have
\begin{equation}
\label{eq:vrsmooth}
\|\tilde D^2 v_r\|_{C^{0,\alpha}(Q_{3/4})} \leq C. 
\end{equation}
This implies that $w_r$ solves the fully nonlinear equation with H\"older coefficients
$$
G(D^2w_r,X)-\partial_t w_r - \partial_t ( v_r + P_r)=(f_r-1)\chi_{A_r} \quad \text{in }Q_{3/4},\qquad G(M,X):=F(D^2 P_r+D^2 v_r(x) +M)-1.
$$
Since $G(0,X)=0$, we can apply  \cite[Theorem 5.6]{Wang1} with $p=n+2$ and \eqref{eq:ABPwr} to obtain
\begin{equation}
\label{eq:wrAr}
\int_{Q_{1/2}}|D^2 w_r|^{n+2} \leq
C\left(\|w_r\|_{L^\infty(Q_{3/4})} + \| \chi_{A_r}\|_{L^{2n}(Q_{3/4})} \right)^{n+2} \leq C\, |A_r|
\end{equation}
(recall that $|A_r|\leq |Q_1|$).

We are now ready to conclude the proof:
since $|\tilde D^2 u_r|\leq K$ a.e. inside $A_r$ (by \eqref{eq:obstacle}),
recalling (\ref{rewrite}) we have
$$
\int_{A_r\cap Q_{1/2}}|\tilde D^2 v_r+  \tilde D^2 w_r + \tilde  
 P_r|^{n+2}
=\int_{A_r \cap Q_{1/2}} |\tilde D^2 u_r|^{n+2} \leq K^{n+2}|A_r|.
$$
Therefore, by \eqref{eq:vrsmooth} and \eqref{eq:wrAr},
\begin{align*}
|A_r\cap Q_{1/2}|\,|\tilde P_r|^{n+2}&=\int_{A_r\cap Q_{1/2}}|\tilde P_r|^{n+2}\\
&\leq 3^{2n}\biggl( \int_{A_r\cap Q_{1/2}}|\tilde D^2 v_r|^{n+2} +  \int_{A_r\cap Q_{1/2}}|\tilde D^2 w_r|^{n+2}
+K^{n+2}|A_r|\biggr)\\
&\leq  3^{n+2} \biggl(|A_r\cap Q_{1/2}|\,\|\tilde D^2 v_r\|_{L^\infty(Q_{1/2})}+ \int_{Q_{1/2}}|\tilde D^2 w_r|^{n+2}+K^{n+2}|A_r|\biggr)\\
&\leq C\,|A_r\cap Q_{1/2}| + C\, |A_r|.
\end{align*}
Hence, if $|\tilde P_r|$ is sufficiently large we obtain
$$
|A_r\cap Q_{1/2}(0)|\,|\tilde P_r|^{n+2} \leq C|A_r| \leq  \frac{1}{4^{n+1}}|\tilde P_r|^{n+2}|A_r|.
$$
Since $|A_{r/2}|= 2^{n+1}|A_r\cap Q_{1/2}(0)|$, this gives the desired result.
\end{proof}


\subsection{Proof of Theorem~\ref{thm:C11}}\label{main-proof}
Taking $M >0$ as in Proposition \ref{prop:M}, we have that one of the following hold:
\begin{enumerate}
    \item[(i)]  $\liminf_{k\to \infty}|P_{2^{-k}}| \leq 3M$,
    \item[(ii)]   $\liminf_{k\to \infty}|P_{2^{-k}}| \geq 3M$.
\end{enumerate}
Then, one consider the two case separately and, arguing exactly as in the proof of Theorem 1.2 in \cite{FS}
one obtains the desired result. (We notice that 
the reference  \cite[Theorem 3]{CaFNL} in that proof is to be replaced by  \cite[Theorem 1.1]{Wang2}.)


\section{Non-degeneracy and global solutions}
\label{sect:non deg}

\subsection{Local non-degeneracy}
As shown in \cite[Section 3]{FS} non-degeneracy fails in general for  the elliptic case, and surely for our problem as well.
 Nevertheless, the non-degeneracy does hold for the case  $\Omega\supset\{\nabla u \neq 0\}$, see \cite[Lemma 3.1]{FS}.
We now show that this non-degeneracy result still holds in the parabolic case:

\begin{lemma}\label{supercond-nondegeneracy}
Let $u:Q_1 \to \R$ be a $W_x^{2,n}\cap W_t^{1,n}$ solution of \eqref{eq:obstacle}, 
assume that $F$ satisfies (H0)-(H2), and that $\Omega\supset\{\nabla u \neq 0\}$.
Then, for any $X^0=(x^0,t^0) \in \overline\Omega\cap Q_{1/2}$,
$$
\max_{\p_p Q_r(X^0)} u \geq u(X^0)+\frac{r^2}{2n\lambda_1 + 1} \qquad \forall\, r \in (0,1/4).
$$
\end{lemma}

\begin{proof}
For
$$
v(x):=u(x)-\frac{|x-x^0|^2- (t-t^0)}{2n\lambda_1+1},
$$
one readily verifies that $\H(v) \geq 0 $ in $Q_r(X^0)$.
Then, by the very same argument as in the proof of \cite[Lemma 3.1]{FS} we deduce that
$$
\max_{\p_p Q_r(X^0)}v =\sup_{Q_r(X^0)}v,
$$
and the result follows easily.
\end{proof}

\subsection{Classification of global solutions}
As already discussed in the previous section, to have non-degeneracy of solutions we need to assume that $\Omega\supset\{\nabla u \neq 0\}$.
In the elliptic case this assumption is also sufficient to classify global solutions with a ``thick free boundary''
(see \cite[Proposition 3.2]{FS}).
However, in the parabolic case the situation is much more complicated: indeed, while global solutions of the elliptic problem
with ``thick free boundary'' are convex and one-dimensional, in the parabolic case we have non-convex solutions.
For instance the function
$$
u(x)=
\left\{
\begin{array}{ll}
-2t-x_1^2/2 & \text{if $x_1>0$},\\
-2t & \text{if $x_1\leq 0$},\\
\end{array}
\right.
$$
is a global solution to the problem $\Delta u -\p_t u=\chi_{\{\nabla u\neq 0\}}$.

In order to avoid these examples, here we shall only consider the case
$\Omega\supset\{u \neq 0\}$. 

Since we will use  minimal diameter to measure sets, we need some facts about their  stability properties:
Let us first recall the definition for $\delta_r(u,X^0)$ for $X^0=(x^0,t^0)$:
$$
\delta_r(u,X^0):=\inf_{t \in [t_0-r^2,t_0+r^2]}\frac{\operatorname{MD}\bigl(\Lambda\cap \bigl(B_r(x^0) \times \{t\}\bigr)\bigr)}{r},\qquad \Lambda :=Q_1\setminus \Omega.
$$
We remark that, for polynomial global solutions $P_2=\sum_j a_j x_j^2 + bt$ (with $A={\rm diag}(a_j)$, and $b$ such that $F(A)-b=1$), one has 
\begin{equation}\label{P2}
\delta_r(P_2,0)=0.
 \end{equation}

Let us also recall the standard scaling and stability estimate
\begin{equation}\label{stability1}
\delta_r(u,X)= \delta_1(u_r,0), \qquad \limsup_{r \to 0} \delta_r(u,X^0)\leq  \delta_1(u_0,0)
 \end{equation}
 whenever $u_r(y,\tau )=u(x+ry, t + r^2\tau)/r^2$ converges uniformly to some function $u_0$.

In the next proposition we classify global solution with a ``thick free boundary''.
We notice that assumption \eqref{min-diam} below allows us to exclude the family of global solutions
$u_\sigma(t,x)=-(t-\sigma)_+$, $\sigma \in \R$.

\begin{proposition}\label{global-convexity}
Let $u:\R^{n+1} \to \R$ be a $W^{2,n}$ solution of \eqref{eq:obstacle} on the whole $\R^{n+1}$, 
assume that $F$ is convex and satisfies (H0)-(H1), and that $\Omega\supset\{u \neq 0\}$.
Furthermore, assume that there exists $\epsilon_0>0$ such that
\begin{equation}\label{min-diam}
\delta_r(u,X^0) \geq \epsilon_0 \qquad \forall\,r>0,\,\forall\,X^0 \in \partial\Omega.
\end{equation}
Then $u$ is time-independent. In particular, by the elliptic case \cite[Proposition 3.2]{FS}, $u$
is a half-space solution, i.e., up to a rotation,
$u(x)=\g [(x_1)_+]^2/2$, where $\g \in (1/\lambda_1,1/\l_0) $ is such that $F(\g e_1\otimes e_1 )=1$.
\end{proposition}

\begin{proof}
  Let $m:=\sup_{\R^{n+1}} \partial_t u $ (notice that $m$ is finite by Theorem \ref{thm:C11}) and
consider a sequence $m_j = \partial_t u(X^j)$ such that $m_j\to m$. 

We now perform the scaling 
    $$u_j(x,t):=\frac{u(d_jx+x^j,d_j^2t+t^j)}{d_j^2},
        $$
       where $X^j=(x^j,t^j)$ and $d_j:={\rm dist}(X^j,\partial\Omega)$.
      
        The functions $u_j$ still satisfy \eqref{eq:obstacle}. Also, since $u=0$ on $\partial\Omega$ it follows by the $C^{1,1}_x\cap C^{0,1}_t$
       regularity of $u$ that $u_j$ are uniformly bounded, hence, up to subsequences, they converge to another global solution $u_\infty$ which satisfies
        $\partial_{t}u_\infty(0)=m$. By \eqref{stability1} and the assumption
        \eqref{min-diam} we obtain
        \begin{equation}\label{min-diam-1}
        \delta_r(u_\infty,X^0) \geq \epsilon_0 \qquad \forall\,r>0,\,\forall\,X^0 \in \partial\Omega_\infty,
        \end{equation}
        where $\Omega_\infty$ is the limit, as $j \to \infty$, of the family of open sets
        $$
        \Omega_j:=\bigl\{(x,t)\,:\,(d_jx+x^j,d_j^2t+t^j) \in \Omega\bigr\}.
        $$
       Let us observe that, by the condition $\Omega\supset\{u \neq 0\}$ we get
        $u_\infty(t,x)=0$ on $\p\Omega_\infty$.
            
        In addition $\partial_{t}u_\infty$ is a solution of the uniformly parabolic linear operator
        $F_{ij}(D^2u_\infty)\partial_{ij} - \partial_t $ inside $\Omega_\infty$.
        Hence, since $\p_t u_\infty\leq m$ and $\partial_{t}u_\infty(0)=m$, by the strong maximum principle we deduce that 
        $\partial_{t}u_\infty$ is constant inside the connected component of $\Omega_\infty$ containing $0$ (call it $\Omega_0$).
  
          Therefore, integrating $u_\infty$ in the direction $t$ gives 
          \begin{equation} 
          \label{eq:formula u0}
           u_\infty(t,x)= mt + U(x) \quad \text{inside $\Omega_0$,} \qquad u_\infty=0\quad \text{on $\p\Omega_0$}.
            \end{equation}
            Thanks to the assumption \eqref{min-diam} it is easy to check that $\p\Omega_0$ cannot be contained into hyperplanes of the form $\{t=\sigma\}$ for $\sigma \in \R$,
            hence \eqref{eq:formula u0} is possible only if $m=0$. Hence we have proved that $\sup_{\R^{n+1}}\partial_t u = 0$.
            
            By a completely symmetric argument we obtain $\inf_{\R^{n+1}} \partial_t u= 0$.
            Thus $\partial_tu=0$, which implies that $u$ is time-independent and therefore, by \cite[Proposition 3.2]{FS}, up to a rotation $u$ is of the form
            $u(x)=\g [(x_1)_+]^2/2+c$ $\g \in (1/\lambda_1,1/\l_0) $ is such that $F(\g e_1\otimes e_1 )=1$ and $c \in \R$.
            Since $\Omega\supset\{u \neq 0\}$ we see that $c=0$, which proves the result.
  \end{proof}


\section{Local solutions and directional monotonicity}\label{dir-mon}

In this section we shall prove a directional monotonicity for solutions to our equations. 
In the next section we will use Lemma~\ref{lemma:monotonicity 1}
below to show that,
if $u$ is close enough to a half-space solution $\g [(x_1)_+]^2$ in a ball $B_r$, then
for any $e=(e_x,e_t) \in \mathbb S^{n}$ with $e \cdot (e_1,0)\geq s>0$ we have $C_0\partial_e u  - u \geq 0 $ inside
$B_{r/2}$.

\begin{lemma}\label{lemma:u_t}
Let $u:Q_1 \to \R$ be a $W_x^{2,n}\cap W_t^{1,n}$ solution of \eqref{eq:obstacle}
with $\Omega\supset\{u \neq 0\}$. Then, under  the conditions of Theorem \ref{thm:C12} 
we have
$$\lim_{\Omega \ni X \to \partial \Omega} \partial_t u (X)=0 . $$
\end{lemma}

\begin{proof}
The proof of this lemma follows easily by a contradictory argument, along with scaling and blow-up. Indeed, given a sequence $X^j \to \partial \Omega$ such that $|\p_tu(X^j)| \geq c>0$,
then one may scale at $X^j$ with $d_j=\hbox{dist} (X^j,\partial \Omega)$ and define
$u_j(X):=\left[u(d_jx + x^j,d_j^2t + t^j) -u(X^j)\right]/d_j^2$ to end up with a global solution $u_\infty$ with the property $\partial_t u_\infty(0) \neq 0 $,
contradicting Proposition \ref{global-convexity}.
\end{proof}

The proof of the following result is a minor modification of the one of \cite[Lemma 4.1]{FS},
so we just give a sketch of the proof.

\begin{lemma}
\label{lemma:monotonicity 1}
Let $u:Q_1 \to \R$ be a $W_x^{2,n}\cap W_t^{1,n}$ solution of \eqref{eq:obstacle}
with $\Omega\supset\{ u \neq 0\}$.
Assume that for some space-time direction $e=(e_x,e_t)$ with $|e|=1$ we have
 $C_0\partial_eu-u \geq -\e_0$ in $Q_1$ for some $C_0,\e_0 \geq 0$,
and that $F$ is convex and satisfies (H0)-(H1).
Then $C_0\partial_eu-u \geq 0$ in $Q_{1/2}$ provided  $\e_0 \leq \frac{1}{4(2n\lambda_1+1)} $.
\end{lemma}

\begin{proof}
Since $F$ is convex, for any matrix $M$ we can choose an element $P^M$ inside $\partial F(M)$
(the subdifferential of  $F$ at $M$)
in such a way that the map $M \mapsto P^M$ is measurable, and we define the measurable uniformly elliptic coefficients
$$
a_{ij}(x,t):=(P^{D^2u(x,t)})_{ij} \in \partial F(D^2u(x,t)).
$$
As in the proof of \cite[Lemma 4.1]{FS}, by the convexity of $F$ if follows that, in the viscosity sense,
\begin{equation}
\label{est:1}
a_{ij}\partial_{ij} (\partial_e u) -\partial_{t} (\partial_e u) \leq 0\qquad \text{in }\Omega
\end{equation}
and
\begin{equation}
\label{est:2}
a_{ij}\partial_{ij}u -\partial_t u   \geq 1\qquad \text{in }\Omega.
\end{equation}

Now, let us assume by contradiction that there exists $X^0=(x^0,t^0)\in Q_{1/2}$ such that $C_0\partial_eu(X^0)-u(X^0)<0$, and consider
the function
$$
w(X):=C_0\partial_eu(X)- u(X)+\frac{|x-x^0|^2-(t-t_0)}{2n\lambda_1+1}.
$$
Thanks to \eqref{est:1}, \eqref{est:2}, and assumption (H1) (which implies that $\lambda_0 \Id \leq a_{ij}\leq \lambda_1 \Id$) we deduce that
$w$ is a supersolution of the linear operator $\Lin:=a_{ij}\partial_{ij} -\partial_t $, hence, by the minimum principle,
$$
\min_{\partial_p (\Omega \cap Q_{1}(Y^0))} w =  
\min_{\Omega \cap Q_{1}(Y^0)} w \leq w(Y^0)<0.
$$

By Lemma \ref{lemma:u_t} and the assumption $\Omega\supset \{u\neq 0\}$ we have
$\partial_t u =u=|\nabla u|= 0$ on $\partial \Omega$, therefore
 $w\geq  0$ on $\partial \Omega$.
Thus, since  $|x-x^0|^2 - (t-t^0)\geq 1/4$ on $\partial_p Q^-_{1}$
 it follows that
$$
0>\min_{\partial_p Q^-_{1/2}(X^0)} w \geq -\e_0+ \frac{1}{4(2n\lambda_1+1)},
$$
a contradiction if $\e_0 \leq \frac{1}{4(2n\lambda_1+1)}$.
\end{proof}

\section{Proof of Theorem~\ref{thm:C12}}\label{Local-regularity}
The proof of this theorem is very similar to the proof of \cite[Theorem 1.3]{FS}.
Indeed, take $X^0=(x^0,t^0) \in \partial \Omega \cap Q_{1/8}$,
and rescale the solution around $X^0$, that is $u_r(x,t):=[u(rx+x^0,r^2t+ t^0)-u(x^0,t^0) - r\nabla u(x^0,t^0)\cdot x]/r^2$.

Because of the uniform $C^{1,1}_x\cap C_t^{0,1}$ estimate provided by Theorem~\ref{thm:C11} and the thickness assumption on the free boundary of $u$,
we can find a sequence $r_j\to 0$
such that $u_{r_j}$ converges  locally uniformly to a 
global solution $u_\infty$ of the form 
$u_\infty(x)=\g[(x\cdot e_{_{X^0}})_+]^2/2 $ with $\gamma \in [1/\l_1,1/\l_0]$ and $e_{_{X^0}} \in \mathbb S^{n-1}$ (see Proposition \ref{global-convexity}).

Notice now that, for any $s \in (0,1)$, we can find a large constant $C_s$ such that
$$
C_s\p_e u_\infty- u_\infty \geq 0\qquad \text{inside $B_1$}
$$
for all directions $e=(e_x,e_t) \in\mathbb S^{n}$ such that $e\cdot (e_{_{X^0}},0) \geq s$,
hence by the $C^1_x$ convergence of $u_{r_j}$ to $u_\infty$ and Lemmas \ref{lemma:u_t} and \ref{lemma:monotonicity 1}
we deduce that
\begin{equation}
\label{eq:monot}
C_s\partial_e  u_{r_j}-  u_{r_j} \geq 0\qquad \text{in $Q_{1/2}$},
\end{equation}
and since $u_{r_j}(0)=0$ a simple ODE argument 
shows that $u_{r_j} \geq 0$ in $Q_{1/4}$.

Using \eqref{eq:monot} again, this implies that 
$\partial_e  u_{r_j}$ inside $Q_{1/4}$,
and so in terms of $u$ we deduce that there exists $r=r(s)>0$ such that
$$
\partial_e u \geq 0 \qquad \text{inside $Q_{r}(X^0)$}
$$
for all $e \in \mathbb S^{n}$ such that $e \cdot (e_{_{X^0}},0)\geq s$.

A simple compactness argument shows that $r$ is independent of the point $x$,
which implies that the free boundary is $s$-Lipschitz.
Since $s$ can be taken arbitrarily small (provided one reduces the size of $r$), this actually
proves that the free boundary is $C^1$.
Higher regularity is then classical.

\section{Appendix: Parabolic BMO estimates}

Let $u:Q_1 \to \R$ satisfy $|u|\leq 1$ and $|\H(u)|\leq M$.
Up to replacing $u$ by $u(x/R, t/R^2)$ and $\H$ by $(F(R^2\ \cdot)/R^2 -\partial_t )$ with $R$ a large fixed constant, we can assume that
$|\H(u)|\leq \delta$ with $\delta$ a small constant to be fixed later. Observe that, with this scaling,  the ellipticity remains the same.

Let us first state a standard stability result.

\begin{lemma}(Compactness)
Let $\e >0$, and  $u$ be such that
 $u:Q_1\to \R$ satisfy $|u|\leq 1$. Let further 
 $v:Q_{1/2}\to \R$ solve
$$
\left\{
\begin{array}{ll}
\H(v)=0& \text{in }Q_{1/2},\\
v=u& \text{on }\partial Q_{1/2}.
\end{array}
\right.
$$
Then  there exists $\delta=\delta (\e)>0$ such that 
$$
|u-v| \leq \e \qquad \text{in }Q_{1/2},
$$
provided $|\H(u)|\leq \delta$.
\end{lemma}

The proof of the lemma is based on a standard compactness argument, using that both
$u$ and $v$ are uniformly   H\"older continuous (in $(x,t)$-variables) inside $Q_{1/2}$;
see \cite{Wang1}, Lemma 5.1.

We say that $P$ is a  ``parabolic'' second order polynomial if it is of the form
$$
P(x,t)=a_0 +\langle b_0,x\rangle + \langle M_0 x,x\rangle + c_0 t,
\qquad a_0,c_0 \in \R,\,b_0 \in \R^n,\,M_0 \in \R^{n\times n}.
$$

We now prove by induction the following result:
\begin{lemma}\label{P_k-close}
  Let  $u:Q_1\to \R$ be a solution to our problem \eqref{eq:obstacle}, with $|u|\leq 1$.
Ten there exists $\rho>0$ universal such that
$$
\left|u(X) - P_k(X) \right| \leq \rho^{2k} \qquad \text{inside $Q_{\rho^k}$} \quad \forall\,k \in \N,
$$
where $P_k$ is a parabolic second order polynomial such that $\H(P_k)=0$.
\end{lemma}

A straight forward implication of this result is that there is a universal constant $C=1/\rho^2$ such that 
\begin{equation}\label{pointwise-bmo}
\left|u(X) - P_r(X) \right| \leq Cr^2 \qquad \text{inside $Q_r$} \quad \forall\,  0<r<1,
\end{equation}
where $P_r$ is a parabolic second order polynomial such that $\H(P_r)=0$. This in turn implies an $L^p$-BMO type result, see the corollary below.

\begin{proof} (of Lemma \ref{P_k-close})
Since the result is obviously true for $k=0$ (just take $P_0=0$), we prove the inductive step.
So, let us assume that the result is true for $k$ and we prove it for $k+1$.

Define $u_k(X):=\frac{u(\rho^kx,\rho^{2k} t)-P_k(\rho^kx,\rho^{2k} t)}{\rho^{2k}}$.
Then, by the inductive hypothesis $|u_k|\leq 1$ inside $Q_1$. In addition
$$
|\H_k(u_k)|\leq \delta ,\qquad \H_k(v):=F(D^2v + D^2P_{k}) - \partial_t P_k- \partial_t v.
$$
Observe that $\H_k$ keeps the same ellipticity as $\H$.
Hence we can apply the lemma above to deduce that 
$$
|u_k-v_k| \leq \e \qquad \text{in }Q_{1/2},
$$
where $v_k$ solves 
$$
\left\{
\begin{array}{ll}
\H_k(v_k)=0& \text{in }Q_{1/2},\\
v_k=u_k& \text{on }\partial Q_{1/2}.
\end{array}
\right.
$$
Since $\|v_k\|_{L^\infty(Q_{1/2})} \leq \|u_k\|_{L^\infty(Q_{1})} \leq 1$,
by interior $C_{\alpha}^{2,1}$ estimates we get
$$
\|v_k\|_{C_\alpha^{2,1}  (Q_{1/4})} \leq C_0.
$$
Let $\hat P_{k}$ be the ``parabolic'' second order Taylor expansion of $v_k$ at $(0,0)$,
and notice that $\H_k(\hat P_k)=\H_k(v_k(0,0))=0$.
Then
$$
|v_k - \hat P_k| \leq C_0\rho^{2+\alpha} \qquad \text{inside $Q_\rho$,}
$$ 
which gives
$$
|u_k - \hat P_k| \leq C_0\rho^{2+\alpha}+\e \qquad \text{inside $Q_\rho$.}
$$
In particular, if we choose $\rho$ sufficiently small so that $C_0\rho^\alpha \leq 1/2$
and then $\e \leq \rho^2/2$ we arrive at
$$
|u_k(X) - \hat P_k(X)| \leq \rho^2 \qquad \text{inside $Q_\rho$},
$$
or equivalently (recalling the definition of $u_k$)
$$
|u(X) - P_{k+1}(X)| \leq \rho^{2(k+1)}, \qquad P_{k+1}(X):=P_k(X) +\rho^{2k}\hat P_k(x/\rho^k,t/\rho^{2k}).
$$
Also, since $\H_k(\hat P_k)=0$  we will have
$$
\H(P_{k+1})=  F(D^2 P_{k+1}) - \partial_t P_{k+1} =
 F( D^2 P_k +  D^2 \hat P_k ) -\partial_t P_k - \partial_t \hat P_k
= \H_k(\hat P_{k})=0 
$$
which concludes the proof of the inductive step.

\end{proof}

\begin{remark}\label{cor:bmo}
As a corollary of our result we deduce $L^p$-BMO estimates on $\tD^2u$ ($p \in (1,\infty)$) for solutions to general elliptic/parabolic operators of type $F=F(\tD^2u, \nabla u,u, X)$
provided $F$ is H\"older continuous and $u \in C^{1,\alpha}_x$.
\end{remark}

Indeed, if $F=F(D^2u,X)$, since $u_k(X):=\frac{u(\rho^k x,\rho^2t)-P_k(\rho^kx,\rho^2t)}{\rho^{2k}}$ satisfies 
$|u_k|\leq 1$ and $|\H_k(u_k)|\leq \delta $ inside $Q_1$,
by interior $W_p^{2,1}$ estimates we get 
$$
\|\tD^2u_k\|_{L^p(Q_{1/2})} \leq C,
$$
that is
$$
\frac{1}{|Q_{\rho^k/2}|} \int_{Q_{\rho^k/2}} |\tD^2 u -\tD^2 P_k|^p \leq C \qquad \forall\,k \in \N.
$$
For general operators $F$ it suffices to apply the above argument to  $G(M,X):=F(M,Du(X),u(X),X)$.


\end{document}